%% file: dyadic.tex
\documentclass{article}
\input{packages}
\input{commands}

\begin{document}

\title{Variation of the dyadic maximal function}
\author{Julian Weigt}
\affil{Department of Mathematics and Systems Analysis, Aalto University, Finland, \texttt{julian.weigt@aalto.fi}}
\maketitle

\begin{abstract}
\input{abstract}

\end{abstract}

\begingroup
\renewcommand\thefootnote{}\footnotetext{%
}%
\addtocounter{footnote}{-1}%
\endgroup

\section{Introduction}

\input{result.tex}

\section{Setup}

\input{strategy.tex}

\section{Proof of \texorpdfstring{\Cref{pro_densitylow}}{Proposition~2.5}}

\input{proof.tex}

\section{Approximating the full maximal operator}\label{sec_approximation}

\input{approximation.tex}

\section{Further approaches}

\input{prospects.tex}

\bibliographystyle{plain}
\bibliography{bib}

\end{document}

%% file: packages.tex
\usepackage{amsmath,amsfonts,amssymb,amsthm,mathtools,enumitem,geometry,esint,color,hyperref,cleveref,authblk}

\crefname{equation}{}{}
\Crefname{equation}{Equation}{Equations}

\theoremstyle{plain}
\newtheorem{theo}{Theorem}[section]

\theoremstyle{definition}

\newtheorem*{exa*}{Example}
\newtheorem{exa}[theo]{Example}
\newtheorem{lem}[theo]{Lemma}
\newtheorem*{lem*}{Lemma}
\newtheorem{cor}[theo]{Corollary}
\newtheorem{pro}[theo]{Proposition}
\newtheorem*{pro*}{Proposition}
\newtheorem{cla}[theo]{Claim}
\newtheorem*{cla*}{Claim}

\newtheorem*{lie*}{Lie}

\theoremstyle{remark}
\newtheorem{rem}[theo]{Remark}
\newtheorem*{rem*}{Remark}

\newtheorem*{note*}{Note}

\allowdisplaybreaks

\crefname{cor}{corollary}{corollaries}
\crefname{pro}{proposition}{propositions}
\crefname{lem}{lemma}{lemmas}

%% file: commands.tex
\newcommand\intd{\mathop{}\!\mathrm{d}}
\newcommand\ind[1]{1_{#1}}
\newcommand\tx\text
\newcommand\f\frac

\newcommand\loc{\mathrm{loc}}
\newcommand\iloc{{\widetilde{\mathrm{loc}}}}

\newcommand\Q{\mathcal Q}
\renewcommand\P{\mathcal P}

\newcommand\avint\fint

\DeclareMathOperator\var{var}

\DeclareMathOperator\sle{l}

\let\div\relax
\DeclareMathOperator\div{div}
\newcommand\M{{\mathrm M}}
\newcommand\Mi{{\widetilde{\mathrm M}}}

\newcommand\BV{\mathrm{BV}}
\newcommand\ub\underbrace

\newcommand\mb[1]{\partial_*{#1}}

\newcommand\mc[1]{\overline{#1}^*}
\newcommand\tc[1]{\overline{#1}}

\newcommand\ti[1]{\mathring{#1}}
\newcommand\sm[1]{\mathcal{H}^{d-1}(#1)}

\newcommand\lm[1]{\mathcal{L}(#1)}
\newcommand\lmb[1]{\mathcal{L}\Bigl(#1\Bigr)}

\newcommand\msl[3]{\lambda[#1,#2,#3]}
\newcommand\trunc[2]{\overline{#2}^{#1}}
\newcommand\p[1]{\mathfrak P(#1)}

%% file: abstract.tex
We prove that for the dyadic maximal operator \(\mathrm M\) and every locally integrable function \(f\in L^1_{\mathrm{loc}}(\mathbb R^d)\) with bounded variation,
also \(\mathrm M f\) is locally integrable
and \(\mathop{\mathrm{var}}\mathrm M f\leq C_d\mathop{\mathrm{var}} f\) for any dimension \(d\geq1\).
It means that if \(f\in L^1_{\mathrm{loc}}(\mathbb R^d)\) is a function whose gradient is a finite measure then so is \(\nabla \mathrm M f\)
and \(\|\nabla \mathrm M f\|_{L^1(\mathbb R^d)}\leq C_d\|\nabla f\|_{L^1(\mathbb R^d)}\).

We also prove this for the local dyadic maximal operator.

%% file: result.tex
Let \(d\in\mathbb{N}\) and \(\Omega\) be an open set in \(\mathbb{R}^d\).
For every locally integrable function \(f\in L^1_\loc(\Omega)\)
we define the dyadic local maximal function by
\[\M_\Omega f(x)=\sup_{x\in Q,\tc Q\subset\Omega}\f1{\lm Q}\int_Q f\]
where the supremum is taken over all dyadic cubes \(Q\) that contain \(x\) and whose closure is contained in \(\Omega\).
With minor modifications
we may also demand \(Q\) or its interior to be contained in \(\Omega\) instead,
see \Cref{rem_mostgeneral}.
Various maximal operators have been investigated.
The most well known are the centered Hardy-Littlewood maximal operator 
which averages over all balls centered in \(x\),
and the uncentered Hardy-Littlewood maximal operator 
which averages over all balls that contain \(x\).

The regularity of a maximal operator was first studied in \cite{MR1469106}, 
where Kinnunen proved for the Hardy-Littlewood maximal operator that for \(p>1\) and \(f\in W^{1,p}(\mathbb{R}^d)\) also the bound
\(\|\nabla\M f\|_p\leq C_{d,p}\|\nabla f\|_p\)
holds, from which it follows that the Hardy-Littlewood maximal operator is bounded on \(W^{1,p}(\mathbb{R}^d)\). 
His proof fails for \(p=1\).
Note that also
\(\|\M f\|_1\leq C_{d,1}\|f\|_1\)
fails for any nonvanishing \(f\in L^1(\mathbb{R}^d)\) because \(\M f\not\in L^1(\mathbb{R}^d)\).
So in 2004 Haj\l{}asz and Onninen asked in \cite{MR2041705} whether for \(f\in W^{1,1}(\mathbb{R}^d)\)
the Hardy-Littlewood maximal function satisfies \(\nabla\M f\in L^1(\mathbb{R}^d)\) and \(\|\nabla\M f\|_1 \leq C_d\|\nabla f\|_1\).
This question for various maximal operators has become a well known problem
and has been subject to lots of research,
but has so far remained essentially unanswered in dimensions larger than one.

Here is the main result of this paper.

\begin{theo}\label{theo_goal}
Let \(d\in\mathbb{N}\) and \(\Omega\) be an open subset of \(\mathbb{R}^d\).
Let \(f\in L^1_\loc(\Omega)\) with \(\var_\Omega f<\infty\).
Then \(\M_\Omega f\in L^1_\loc(\Omega)\) and
\[\var_\Omega\M_\Omega f\leq C_d\var_\Omega f\]
where \(C_d\) depends only on \(d\).
\end{theo}

Note that we can for example take \(\Omega=\mathbb{R}^d\).
\Cref{theo_goal} answers the question of Haj\l{}asz and Onninen for the dyadic maximal operator in the appropriate sense
because it is clear that the gradient of the dyadic maximal function usually does not exist as a function in \(L^1(\mathbb{R}^d)\).
It also means that \(\|\nabla\M f\|_p\leq C_{d,p}\|\nabla f\|_p\) does not make sense for any \(p\) for the dyadic maximal operator.
However for \(p=1\)
it actually suffices also for the uncentered Hardy-Littlewood maximal operator to prove \(\var\M f\leq C_{d,p}\var f\),
because under this assumption \(f\in W^{1,1}(\mathbb{R}^d)\) implies \(\nabla\M f\in L^1(\mathbb{R}^d)\).
This is due to Panu Lahti \cite{panubvsobolev}.
In this sense \Cref{theo_goal} is the first full answer to the question of Haj\l{}asz and Onninen for any maximal operator in dimensions larger than one to the best of our knowledge.

In \cite{weigt2020variation} we already proved \Cref{theo_goal} for characteristic functions
for the dyadic and the uncentered Hardy-Littlewood maximal operator.
This paper also makes use of \Cref{lem_varav},
which is a variant of the relative isoperimetric inequality established in \cite{weigt2020variation}.

In one dimension for \(L^1(\mathbb{R})\) the gradient bound has already been proven in \cite{MR1898539} by Tanaka for the uncentered maximal function,
and later in \cite{MR3310075} by Kurka for the centered Hardy-Littlewood maximal function. 
The latter proof turned out to be much more complicated. 
In \cite{MR2276629}, Aldaz and P\'erez L\'azaro improved Tanaka's bound to the sharp
\(\|\nabla\M f\|_{L^1(\mathbb{R})}\leq\|\nabla f\|_{L^1(\mathbb{R})}\)
for the uncentered Hardy-Littlewood maximal function.
In \cite{MR3800463} Luiro has proven the gradient bound for the uncentered maximal operator for radial functions in \(W^{1,1}(\mathbb{R}^d)\).
In \cite{MR2539555} Aldaz and P\'erez L\'azaro have done the same for block decreasing functions.

As a first step towards weak differentiability, Haj\l{}asz and Mal\'y proved in \cite{MR2550181}
that for \(f\in L^1(\mathbb{R}^d)\) the centered Hardy-Littlewood maximal function is approximately differentiable.
In \cite{MR2868961} Aldaz, Colzani and P\'erez L\'azaro prove bounds on the modulus of continuity for all dimensions.
A related question is whether the maximal operator is a continuous operator.
Luiro proved in \cite{MR2280193} that for \(p>1\) the uncentered maximal operator is continuous on \(W^{1,p}(\mathbb{R}^d)\).
There is ongoing research for the endpoint case \(p=1\).
For example Carneiro, Madrid and Pierce proved in \cite{MR3695894} that for the uncentered maximal function
\(f\mapsto\nabla\M f\) is continuous \(W^{1,1}(\mathbb{R})\rightarrow L^1(\mathbb{R})\).

The regularity of maximal operators has also been studied on other spaces and for other maximal operators.
We focus on the endpoint \(p=1\).
For example in \cite{MR3063097} Carneiro and Svaiter and in \cite{carneiro2019gradient} and Carneiro and Gonz\'alez-Riquelme consider
convolution maximal operators associated to certain partial differential equations.
They prove
\(\|\nabla\M f\|_{L^1(\mathbb{R}^d)}\leq C_d\|\nabla f\|_{L^1(\mathbb{R}^d)}\)
for \(d=1\), and for \(d>1\) if \(f\) is radial.
In \cite{MR3091605} Carneiro and Hughes proved the discrete result
\(\|\nabla\M f\|_{l^1(\mathbb{Z}^d)}\leq C_d\|f\|_{l^1(\mathbb{Z}^d)}\)
for centered and uncentered maximal operators.
This bound does not hold on \(\mathbb{R}^d\)
but is weaker than the yet unknown
\(\|\nabla\M f\|_{l^1(\mathbb{Z}^d)}\leq C_d\|\nabla f\|_{l^1(\mathbb{Z}^d)}\),
due to
\(\|\nabla f\|_{l^1(\mathbb{Z}^d)}\leq C_d\|f\|_{l^1(\mathbb{Z}^d)}\).
In \cite{MR2328816} Kinnunen and Tuominen work in the metric setting.
They prove the boundedness of a discrete maximal operator in the Haj\l{}asz Sobolev space \(M^{1,1}\).
In \cite{MR3891939} Pérez, Picon, Saari and Sousa consider Hardy-Sobolev spaces instead of Sobolev spaces.
They prove the boundedness of certain convolution maximal operators on \(\dot H^{1,p}\) for a sharp range of exponents, including \(p=1\).
The study of the regularity of the fractional maximal operators was initiated by Kinnunen and Saksman in \cite{MR1979008}.
It does not map from \(L^p(\mathbb{R}^d)\) to \(L^p(\mathbb{R}^d)\) but the exponent changes, so also the endpoint question,
formulated in \cite{MR3624402}, looks a little different.
It remains unanswered, but there is partial progress, similarly as for the Hardy-Littlewood maximal operator;
see for example \cite{beltran2019regularity,MR3912794,MR3624402,2017arXiv171007233L}.
For more background information on the regularity of maximal operators
there is a survey \cite{carneiro2019regularity} by Carneiro.

The dyadic maximal operator has enjoyed a bit less attention than its continuous counterparts,
such as the centered and the uncentered Hardy-Littlewood maximal operator.
Dyadic cubes are usually way easier to deal with than balls,
but the dyadic version still serves as a model case for the continuous versions since they share many properties.
A classical example is the Hardy-Littlewood maximal inequality,
where the proofs are identical for both types of maximal operators,
after the Vitali Covering Lemma is applied for the continuous version,
which however is the most complicated part of the proof.
Another example is \cite{weigt2020variation},
which proves \(\var\M\ind E\leq C_d\var\ind E\) for the dyadic maximal operator and the uncentered Hardy-Littlewood maximal operator.
The proof for the dyadic maximal operator is much easier,
but the same proof strategy also works for the uncentered maximal operator.
But also there, the general strategy is simple in comparison to the tools that are needed to apply the strategy in the continuous setting.
Therefore this paper may raise hope that the variation boundedness also holds for continuous maximal operators.
For the centered Hardy-Littlewood maximal operator another strategy would be necessary though,
because both the proofs in \cite{weigt2020variation} and here rely on the fact that
the levels sets \(\{\M f>\lambda\}\) of the maxmial functions can be written as
the union of all balls/dyadic cubes \(X\) with \(\int_X f>\lambda\lm X\),
which does not hold for the centered Hardy-Littlewood maximal function.

For \(f\in L^1_\loc(\Omega)\)
it already follows from well-known theory that
\(\M_\Omega f\in L^1_\loc(\Omega)\).
We state it in \Cref{theo_goal} because it is a prerequisite to define the variation of \(\M_\Omega f\).
Define \(L^1_\iloc(\Omega)\) to be the set of all functions \(f\) such that for each measurable and bounded set \(U\)
we have that \(\int_U|f|\) is finite.
Note that \(L^1_\iloc(\Omega)\subset L^1_\loc(\Omega)\).
For \(f\in L^1_\iloc(\Omega)\) define
\[\Mi_\Omega f(x)=\sup_{x\in Q,\ \ti Q\subset\Omega}\f1{\lm Q}\int_Qf,\]
where the supremum is taken over all dyadic cubes \(Q\) that contain \(x\) and whose interior is contained in \(\Omega\).
For local maximal operators such as \(\Mi_\Omega\),
\(L^1_\loc(\Omega)\) is not the correct domain of definition
because \(f\in L^1_\loc(\Omega)\) does not imply that \(\Mi_\Omega f\) is finite almost everywhere.
This has already been observed in footnote (2) of \cite[p.~170]{MR2041705}.
Instead the following variant of \Cref{theo_goal} holds true.

\begin{theo}\label{theo_goalint}
Let \(f\in L^1_\iloc(\Omega)\) with \(\var_\Omega f<\infty\).
Then \(\Mi_\Omega f\in L^1_\iloc(\Omega)\) and
\[\var_\Omega\Mi_\Omega f\leq C_d\var_\Omega f.\]
\end{theo}

\begin{rem}
\Cref{theo_goal} and \Cref{theo_goalint} also extend to the maximal function of the absolute value due to
\(\var\M_\Omega(|f|)\leq C_d\var_\Omega|f|\leq C_d\var_\Omega f\).
\end{rem}
Typically, the maximal operator integrates over \(|f|\) instead of \(f\).
That is because traditionally the maximal function is used for \(L^p\) estimates
for which the absolute value of a function matters.
However here we are looking at regularity properties
and didn't see a major reason to restrict like that.

\begin{rem}\label{rem_mostgeneral}
As will be visible from the proof,
\Cref{theo_goal} and \Cref{theo_goalint} actually hold for all maximal operators of the form
\[\M f(x)=\sup_{x\in Q,\ Q\in\Q}\f1{\lm Q}\int_Qf,\]
where \(\Q\) is a collection of dyadic cubes \(Q\) with \(\tc Q\subset \Omega\) or \(\ti Q\subset\Omega\) respectively,
and functions \(f\) such that \(\M f\geq f\) a.e.\ in \(\Omega\).
The constant is the same as in \Cref{theo_goal},
in particular it only depends on \(d\).
\end{rem}

\begin{rem}
The discrete version of \Cref{theo_goal} also holds on \(\mathbb{Z}^d\).
This is a consequence of the correspondence between a function \(f\) on \(\mathbb{Z}^d\)
and \(\hat f=\sum_{z\in\mathbb{Z}^d}f(z)\ind{[0,1]^d+z}\) on \(\mathbb{R}^d\).
\end{rem}

The main step towards the proof of \Cref{theo_goal} is the following finite version.

\begin{pro}\label{pro_goal}
Let \(\Q\) be a finite set of dyadic cubes,
such that for each dyadic cube \(P\subset\bigcup_{Q\in\Q}Q\) for which there is a \(Q\in\Q\) with \(Q\subset P\)
we have \(P\in\Q\).
Let \(\Omega\) be open with
\(\bigcup_{Q\in\Q}\ti Q\subset\Omega.\)
Let \(f\in L^1(\Omega)\) and denote
\[\M_\Q f(x)=\max\Bigl\{f(x),\f1{\lm Q}\int_Qf:x\in Q\in\Q\Bigr\}.\]
Then
\begin{equation}\label{eq_goalfinite}
\var_\Omega\M_\Q f\leq C_d\var_\Omega f.
\end{equation}
\end{pro}

We first prove \Cref{pro_goal} because it allows us to set aside convergence issues.
In the proof of \Cref{theo_goal} we only use \Cref{pro_goal} with \(\Omega=\bigcup_{Q\in\Q}\ti Q\).
\begin{rem}
In \Cref{pro_goal} we could also prove \(\var_U\M_\Q f\leq C_d\var_U f\) for any Borel set \(U\subset\Omega\) with \(\bigcup_{Q\in\Q}\ti Q\subset U\)
because as noted in \cite[Theorem~3.40]{MR1857292}, the coarea formula \Cref{lem_coareabv} also holds for Borel sets.
\end{rem}

I would like to thank my supervisor, Juha Kinnunen for all of his support,
Panu Lahti for repeated reading of and advice on the manuscript,
and Olli Saari for his idea on how to prove \(\M_\Omega f\in L^1_\loc(\Omega)\) more quickly.
The author has been supported by the Vilho, Yrj\"o and Kalle V\"ais\"al\"a Foundation of the Finnish Academy of Science and Letters.

%% file: strategy.tex
We work in the setting of functions of bounded variation, as in Evans-Gariepy \cite{MR3409135}, Section~5. 
For an open set \(\Omega\subset\mathbb{R}^d\), a function \(f\in L^1_\loc(\Omega)\) is said to have locally bounded variation
if for each open and compactly supported \(V\subset \Omega\) we have
\[\sup\Bigl\{\int_Vf\div\varphi:\varphi\in C^1_{\tx c}(V;\mathbb{R}^d),\ |\varphi|\leq1\Bigr\}<\infty.\]
Such a function comes with a measure \(\mu\) and a function \(\nu:\Omega\rightarrow\mathbb{R}^d\) that has \(|\nu|=1\) \(\mu\)-a.e.\ such that for all \(\varphi\in C^1_{\tx c}(\Omega;\mathbb{R}^d)\) we have
\[\int_Vf\div\varphi=\int_V\varphi\nu\intd \mu.\]
We define the variation of \(f\) in \(\Omega\) by
\[\var_\Omega f=\mu(\Omega).\]
Recall the definition of the set of dyadic cubes
\[\bigcup_{n\in\mathbb{Z}}\{[x_1,x_1+2^n)\times\ldots\times[x_d,x_d+2^n):i=1,\ldots,n,\ x_i\in2^n\mathbb{Z}\}.\]
For a dyadic cube \(Q\) denote by \(\sle(Q)\) the sidelength of \(Q\).
For a locally integrable function \(f\) denote
\[f_Q=\avint_Qf=\f1{\lm Q}\int_Qf.\]
For a set \(\Q\) of dyadic cubes denote
\[\bigcup\Q=\bigcup_{Q\in\Q}Q\]
as is commonly used in set theory.
For a set \(\Omega\subset\mathbb{R}^d\) denote by \(\p \Omega\) the set of dyadic cubes contained in \(\Omega\).
By \(a\lesssim b\) we mean that there exists a constant \(C_d\) that depends only on the dimension \(d\) such that \(a\leq C_d b\).
For a measurable set \(E\subset\mathbb{R}^d\) we define the measure theoretic boundary by
\begin{align*}
\mb E&=\Bigl\{x:\limsup_{r\rightarrow0}\f{\lm{B(x,r)\setminus E}}{r^d}>0,\ \limsup_{r\rightarrow0}\f{\lm{B(x,r)\cap E}}{r^d}>0\Bigr\}\\
\intertext{and the measure theoretic closure by}
\mc E&=\Bigl\{x:\limsup_{r\rightarrow0}\f{\lm{B(x,r)\cap E}}{r^d}>0\Bigr\}.
\end{align*}
We denote the topological interior, boundary and closure by \(\ti E,\ \partial E,\ \tc E\).
Note that for finite unions of cubes the measure theoretic boundary, closure and interior agree with the respective topological quantities.

As in \cite{weigt2020variation}, our approach to the variation is the coarea formula.

\begin{lem}[Theorem~3.40 in \cite{MR1857292}]\label{lem_coareabv}
Let \(\Omega\subset\mathbb{R}^d\) be open.
Let \(f\in L^1_\loc(\Omega)\).
Then
\[\var_\Omega f=\int_\mathbb{R}\sm{\mb{\{f>\lambda\}\cap \Omega}}\intd\lambda.\]
\end{lem}

We need the following elementary decomposition of the measure theoretic boundary of the union of two sets.

\begin{lem}[Lemma~1.7 in \cite{weigt2020variation}]\label{lem_boundaryofunion}
Let \(A,B\subset\mathbb{R}^d\) be measurable. 
Then
\[\mb{(A\cup B)}\subset\mb A\setminus\mc B\cup\mb B\setminus\mc A\cup(\mb A\cap\mb B).\]
\end{lem}

The proof of \Cref{lem_boundaryofunion} is straightforward and can be found in \cite{weigt2020variation}.

A central tool is the relative isoperimetric inequality.
In Theorem~5.11 in \cite{MR3409135} it is stated for balls,
but it also holds for cubes, see Theorem~107 in \cite{lecturenoteshajlasz}.

\begin{lem}\label{lem_rii}
Let \(Q,E\subset\mathbb{R}^d\) be a cube and a measurable set with \(\lm{E\cap Q}\leq\f12\lm Q\).
Then
\[\lm{E\cap Q}^{d-1}\lesssim\sm{\mb E\cap\ti Q}^d.\]
\end{lem}

The following result from \cite{weigt2020variation} is closely related to the relative isoperimetric inequality.

\begin{lem}[Proposition~3.1 in \cite{weigt2020variation}]\label{lem_varav}
Let \(E\subset\mathbb{R}^d\) be measurable and \(Q\) a cube (or a ball) with \(\lm{E\cap Q}=\lambda\lm Q\). 
Then
\[\sm{\partial Q\setminus\mc E}\lesssim\lambda^{-\f{d-1}d}\sm{\mb E\cap\ti Q}.\]
\end{lem}

In the proof of \Cref{pro_goal} we split the variation of \(\M_\Q f\) into two pieces.
One piece can be bounded using \Cref{lem_varav}.
Bounding the second piece is the main contribution of this paper.
We formulate it as follows.

\begin{pro}\label{pro_densitylow}
Let \(f\) and \(\Q\) be as in \Cref{pro_goal}.
For each \(\lambda\in\mathbb{R}\) let \(\Q_\lambda\) be the set of maximal cubes of \(\{Q\in\Q:f_Q>\lambda\}\).
Then
\[\sum_{Q\in\Q}\int_{\lambda:Q\in\Q_\lambda,\ \lm{Q\cap\{f>\lambda\}}<2^{-d-2}\lm{Q}}\sm{\partial Q}\intd\lambda\lesssim\var_{\bigcup\{\ti Q:Q\in\Q\}} f\]
\end{pro}

\begin{proof}[Proof of \Cref{pro_goal}]
For each \(\lambda\), denote by \(\Q_\lambda\) 
the set of maximal cubes in \(\Q\) with \(f_Q>\lambda\).
Then \(\Q_\lambda\) consists of disjoint cubes and
\[\{\M_\Q f>\lambda\}=\bigcup\Q_\lambda\cup\{f>\lambda\}.\]
By \Cref{lem_coareabv,lem_boundaryofunion} we get
\begin{align*}
\var_\Omega\M_\Q f&=\int_{-\infty}^\infty\sm{\mb{\{\M_Q f>\lambda\}}\cap \Omega}\intd\lambda\\
&=\int_{-\infty}^\infty\sm{\mb{(\bigcup\Q_\lambda\cup\{f>\lambda\})}\cap \Omega}\intd\lambda\\
&\leq\int_{-\infty}^\infty\sm{\mb{\bigcup\Q_\lambda}\setminus\mc{\{f>\lambda\}}\cap \Omega}+\sm{\mb{\{f>\lambda\}}\cap \Omega}\intd\lambda\\
&\leq\int_{-\infty}^\infty\sm{\mb{\bigcup\Q_\lambda}\setminus\mc{\{f>\lambda\}}}\intd\lambda+\var_\Omega f.
\end{align*}
It remains to estimate the first summand.
We split it into two parts.
\begin{align*}
\int_{-\infty}^\infty\sm{\mb{\bigcup\Q_\lambda}\setminus\mc{\{f>\lambda\}}}\intd\lambda
&\leq\int_{-\infty}^\infty
\sum_{Q\in\Q_\lambda}\sm{\partial Q\setminus\mc{\{f>\lambda\}}}
\intd\lambda\\
&=\sum_{Q\in\Q}\int_{\lambda:Q\in\Q_\lambda}\sm{\partial Q\setminus\mc{\{f>\lambda\}}} \intd\lambda\\
&=\sum_{Q\in\Q}\int_{\lambda:Q\in\Q_\lambda,\ \lm{Q\cap\{f>\lambda\}}\geq2^{-d-2}\lm{Q}}\sm{\partial Q\setminus\mc{\{f>\lambda\}}} \intd\lambda\\
&+\sum_{Q\in\Q}\int_{\lambda:Q\in\Q_\lambda,\ \lm{Q\cap\{f>\lambda\}}<2^{-d-2}\lm{Q}}\sm{\partial Q\setminus\mc{\{f>\lambda\}}} \intd\lambda
\end{align*}

The second summand in the previous display is bounded by \Cref{pro_densitylow}.
The first summand can be bounded using \Cref{lem_varav}.

\begin{align*}
&\sum_{Q\in\Q}\int_{\lambda:Q\in\Q_\lambda,\ \lm{Q\cap\{f>\lambda\}}\geq2^{-d-2}\lm{Q}}\sm{\partial Q\setminus\mc{\{f>\lambda\}}}\intd\lambda\\
&\lesssim \sum_{Q\in\Q}\int_{\lambda:Q\in\Q_\lambda\ \lm{Q\cap\{f>\lambda\}}\geq2^{-d-2}\lm{\ti Q}}\sm{\mb{\{f>\lambda\}}\cap Q}\intd\lambda\\
&\leq \int_{-\infty}^\infty\sum_{\lambda:Q\in\Q_\lambda}\sm{\mb{\{f>\lambda\}}\cap\ti  Q}\intd\lambda\\
&\leq \int_{-\infty}^\infty\sm{\mb{\{f>\lambda\}}\cap\bigcup\{\ti Q:Q\in\Q\}}\intd\lambda\\
&=\var_{\bigcup\{\ti Q:Q\in\Q\}} f
\end{align*}

\end{proof}

%% file: proof.tex
For a finite set \(\Q\) of dyadic cubes, a dyadic cube \(Q\) and \(r\geq0\) set
\[
\msl Q\Q r
=
\min\biggl\{
\max\Bigl\{
\inf\{\lambda:\lm{\{f>\lambda\}\cap Q}<r\lm Q\}
,
\max\{f_P:P\in\Q,\ Q\subsetneq P\}
\Bigr\}
,f_Q
\biggr\}
.
\]
Then \Cref{pro_densitylow} claims
\[
\sum_{Q\in\Q}
\Bigl(
f_Q-\msl Q\Q{2^{-d-2}}
\Bigr)
\sm{\partial Q}
\intd\lambda
\lesssim
\var_{\bigcup\{\ti Q:Q\in\Q\}} f
.
\]
The part
\(
\inf
\{\lambda:\lm{\{f>\lambda\}\cap Q}<r\lm Q\}
\)
is also called the \(r\)-median of \(f\) on \(Q\).

Recall that \(\p\Omega\) denotes the set of dyadic cubes contained in \(\Omega\).
\begin{pro}\label{cla_mostmasssparse}
Let \(Q_0\) be a dyadic cube, \(\lambda_0\in\mathbb{R}\)
and \(f\in L^1(Q_0)\) with \(\lm{\{f>\lambda_0\}\cap Q_0}\leq2^{-d-1}\lm{Q_0}\).
For each \(\lambda\in\mathbb{R}\) denote by \(\Q_\lambda\) the set of all maximal cubes \(Q\subset Q_0\) with \(f_Q>\lambda\).
Then
\[
\lm{Q_0}(f_{Q_0}-\lambda_0)
\leq2^{d+1}
\int_{\lambda_0}^\infty
\lmb{
\{f>\lambda\}
\cap
\bigcup\bigl\{P\in\Q_\lambda:
\lm{\{f>\lambda\}\cap P}\leq\lm P/2
\bigr\}
}
\intd\lambda
.
\]
\end{pro}

\begin{proof}
Recall that
\begin{equation}\label{eq_Qlambdacoverssuperlevelset}
\int_{\lambda_0}^\infty\lm{\{f>\lambda\}}\intd\lambda=\int_{\lambda_0}^\infty\lmb{\{f>\lambda\}\cap\bigcup\Q_\lambda}\intd\lambda.
\end{equation}

Let \(\P\) be the set of dyadic cubes \(Q\subset Q_0\) with
\(\lm{\{f>f_Q\}\cap Q}>2^{-d-1}\lm Q\)
and denote by \(\tilde\P\) the set of maximal cubes of \(\P\).
Then
\begin{align}
\nonumber&\int_{\lambda_0}^\infty\lmb{\{f>\lambda\}\cap\bigcup\P}\intd\lambda
-\int_{-\infty}^{\lambda_0}\lmb{\{f<\lambda\}\cap\bigcup\P}\intd\lambda\\
\nonumber&=\sum_{Q\in\tilde\P}\biggl(\int_{\lambda_0}^\infty\lm{\{f>\lambda\}\cap Q}\intd\lambda
-\int_{-\infty}^{\lambda_0}\lm{\{f<\lambda\}\cap Q}\intd\lambda\biggr)\\
\nonumber&=\sum_{Q\in\tilde\P}\lm Q(f_Q-\lambda_0)\\
\nonumber&\leq\sum_{Q\in\tilde\P,f_Q>\lambda_0}\lm Q(f_Q-\lambda_0)\\
\nonumber&\leq2^{d+1}\sum_{Q\in\tilde\P,f_Q>\lambda_0}\lm{\{f>f_Q\}\cap Q}(f_Q-\lambda_0)\\
\nonumber&\leq2^{d+1}\sum_{Q\in\tilde\P,f_Q>\lambda_0}\int_{\lambda_0}^{f_Q}\lm{\{f>\lambda\}\cap Q}\intd\lambda\\
\label{eq_Qhasmuchmassinside}&=2^{d+1}\int_{\lambda_0}^\infty\lmb{\bigcup_{Q\in\tilde\P,f_Q\geq\lambda}\{f>\lambda\}\cap Q}\intd\lambda.
\end{align}
In the last equality we interchanged the order of summation and integration and used the disjointness of \(\tilde\Q\).
By \cref{eq_Qhasmuchmassinside} and \cref{eq_Qlambdacoverssuperlevelset} we get
\begin{align*}
\lm{Q_0}(f_{Q_0}-\lambda_0)
&=\int_{\lambda_0}^\infty\lm{\{f>\lambda\}}\intd\lambda
-\int_{-\infty}^{\lambda_0}\lm{\{f<\lambda\}}\intd\lambda\\
&=\int_{\lambda_0}^\infty\lmb{\{f>\lambda\}\cap\bigcup\P}\intd\lambda
-\int_{-\infty}^{\lambda_0}\lmb{\{f<\lambda\}\cap\bigcup\P}\intd\lambda\\
&+\int_{\lambda_0}^\infty\lmb{\{f>\lambda\}\setminus\bigcup\P}\intd\lambda
-\int_{-\infty}^{\lambda_0}\lmb{\{f<\lambda\}\setminus\bigcup\P}\intd\lambda\\
&\leq2^{d+1}\int_{\lambda_0}^\infty\lmb{\bigcup_{Q\in\tilde\P,f_Q\geq\lambda}\{f>\lambda\}\cap Q}\intd\lambda
+\int_{\lambda_0}^\infty\lmb{\{f>\lambda\}\setminus\bigcup\P}\intd\lambda\\
&\leq2^{d+1}\int_{\lambda_0}^\infty\lmb{\{f>\lambda\}\setminus\bigcup\tilde\P\cup\bigcup_{Q\in\tilde\P,f_Q\geq\lambda}\{f>\lambda\}\cap Q}\intd\lambda\\
&=2^{d+1}\int_{\lambda_0}^\infty\lmb{\{f>\lambda\}\setminus\bigcup\{Q\in\tilde\P,f_Q<\lambda\}}\intd\lambda\\
&=2^{d+1}\int_{\lambda_0}^\infty\lmb{\{f>\lambda\}\cap\bigcup\Q_\lambda\setminus\bigcup\{Q\in\tilde\P,f_Q<\lambda\}}\intd\lambda\\
&\leq2^{d+1}\int_{\lambda_0}^\infty\lmb{\{f>\lambda\}\cap\bigcup\{Q\in\Q_\lambda:\lnot\exists P\in\tilde\P\ Q\subsetneq P\}}\intd\lambda
.
\end{align*}
It remains to show that
if \(Q\in\P\) such that for all \(P\in\tilde\P\) we do not have \(Q\subsetneq P\),
then for all
\(
\max\bigl\{
\lambda_0,
\max\{
f_P:P\subset Q_0,P\supsetneq Q
\}
\bigr\}
<\lambda<f_Q
\)
we have
\(\lm{\{f>f_Q\}\cap Q}\leq\lm Q/2\).
If \(Q=Q_0\) then this is true by assumption.
If \(Q\subsetneq Q_0\) then the dyadic parent \(P\) of \(Q\) is contained in \(Q_0\)
and not in \(\P\),
and since \(Q\in\Q_\lambda\) we have \(\lambda>f_P\).
Therefore
\begin{align*}
\lm{\{f>\lambda\}\cap Q}
&\leq\lm{\{f>\lambda\}\cap P}\\
&\leq\lm{\{f>f_P\}\cap P}\\
&\leq2^{-d-1}\lm P\\
&=\f{\lm Q}2.
\end{align*}
\end{proof}

\begin{rem}
By the Lebesgue differentiation theorem we have that \(\bigcup\Q\) contains almost all of \(Q_0\).
Thus some terms in the proof of \Cref{cla_mostmasssparse} are actually \(0\).
\end{rem}

\begin{cor}\label{cla_mostmasssparseabove}
Let \(Q_0\) be a dyadic cube, \(\lambda_0\in\mathbb{R}\)
and \(f\in L^1(Q_0)\) with \(\lm{\{f>\lambda_0\}\cap Q_0}\leq2^{-d-2}\lm{Q_0}\).
Then
\[
\lm{Q_0}
(f_{Q_0}-\lambda_0)
\leq2^{d+2}
\sum_{P\subsetneq Q_0}
\int_{\msl P{\p{Q_0}}{1/2}}^{f_P}\lm{P\cap\{f>\lambda\}}\intd\lambda
.
\]
\end{cor}

\begin{proof}
We have
\[
\int_{\lambda_0}^{f_{Q_0}}
\lm{\{f>\lambda\}\cap Q_0}
\intd\lambda
\leq
\int_{\lambda_0}^{f_{Q_0}}2^{-d-2}
\lm{Q_0}
\intd\lambda
=2^{-d-2}
\lm{Q_0}
(f_{Q_0}-\lambda_0)
\]
so that we get
\[\lm{Q_0}(f_{Q_0}-\lambda_0)\leq2\Bigl[\lm{Q_0}(f_{Q_0}-\lambda_0)-2^{d+1}\int_{\lambda_0}^{f_{Q_0}}\lm{\{f>\lambda\}\cap Q_0}\intd\lambda\Bigr].\]
Since \(\lm{\{f>\lambda_0\}\cap Q_0}\leq2^{-d-2}\lm{Q_0}<2^{-d-1}\lm{Q_0}\) we can apply \Cref{cla_mostmasssparse}
to estimate \(\lm{Q_0}(f_{Q_0}-\lambda_0)\) on the right hand side of the previous display and obtain
\begin{align*}
\lm{Q_0}(f_{Q_0}-\lambda_0)
&\leq2\Bigl[\lm{Q_0}(f_{Q_0}-\lambda_0)-2^{d+1}\int_{\lambda_0}^{f_{Q_0}}\lm{\{f>\lambda\}\cap Q_0}\intd\lambda\Bigr]\\
&\leq2
\Bigl[
2^{d+1}
\int_{\lambda_0}^\infty
\lmb{
\{f>\lambda\}
\cap
\bigcup\bigl\{P\in\Q_\lambda:
\lm{\{f>\lambda\}\cap P}\leq\lm P/2
\bigr\}
}
\intd\lambda
\\
&\qquad
-
2^{d+1}\int_{\lambda_0}^{f_{Q_0}}\lm{\{f>\lambda\}\cap Q_0}\intd\lambda
\Bigr]\\
&\leq2^{d+2}
\int_{f_{Q_0}}^\infty
\lmb{
\{f>\lambda\}
\cap
\bigcup\bigl\{P\in\Q_\lambda:
\lm{\{f>\lambda\}\cap P}\leq\lm P/2
\bigr\}
}
\intd\lambda
\\
&=2^{d+2}
\int_{f_{Q_0}}^\infty
\sum_{P\in\Q_\lambda:
\lm{\{f>\lambda\}\cap P}\leq\lm P/2
}
\lm{
P
\cap
\{f>\lambda\}
}
\intd\lambda
\\
&=2^{d+2}
\int_{f_{Q_0}}^\infty
\sum_{P\subsetneq Q_0:
\msl P{\p{Q_0}}{1/2}<\lambda<f_P
}
\lm{
P
\cap
\{f>\lambda\}
}
\intd\lambda
\\
&=2^{d+2}
\sum_{P\subsetneq Q_0}
\int_{\msl P{\p{Q_0}}{1/2}}^{f_P}
\lm{P\cap\{f>\lambda\}}
\intd\lambda
.
\end{align*}
\end{proof}

\begin{proof}[Proof of \Cref{pro_densitylow}]
Let \(Q\in\Q\) with \(\msl Q\Q{2^{-d-2}}<f_Q\).
Then every \(\lambda>\msl Q\Q{2^{-d-2}}\) satisfies the premise of \Cref{cla_mostmasssparseabove}
so that
\begin{align*}
\int_{\lambda:Q\in\Q_\lambda,\ \lm{Q\cap\{f>\lambda\}}<2^{-d-2}\lm Q}\sm{\partial Q}\intd\lambda
&=
(f_Q-\msl Q\Q{2^{-d-2}})\sm{\partial Q}
\\
&\leq
\f{d2^{d+2}}{\sle(Q)}
\sum_{P\subsetneq Q}
\int_{\msl P{\p Q}{1/2}}^{f_P}\lm{P\cap\{f>\lambda\}}\intd\lambda
.
\end{align*}
By \(\lm{\{f>\lambda\}\cap P}\leq\lm{P}/2\) and \Cref{lem_rii} we have
\[
\f{\lm{\{f>\lambda\}\cap P}}{\sle(P)}
\leq
\lm{\{f>\lambda\}\cap P}^{1-\f1d}
\lesssim
\sm{\partial\{f>\lambda\}\cap\ti P}
.
\]
Further note that
\[
\sum_{Q\supsetneq P}\f1{\sle(Q)}
=
\f1{\sle(P)}
.
\]
Combining the previous three displays,
using \(\msl P{\p Q}{1/2}=\msl P{\p{\bigcup\Q}}{1/2}\) which follows from the assumption on \(\Q\),
and using that for \(P_2\subsetneq P_1\subsetneq Q\in\Q\)
we have \(\msl{P_2}{\p{\bigcup\Q}}{1/2}\geq\min\{f_{P_1},f_{P_2}\}\),
we can conclude
\begin{align*}
&
\sum_{Q\in\Q}
\int_{\lambda:Q\in\Q_\lambda,\ \lm{Q\cap\{f>\lambda\}}<2^{-d-2}\lm{Q}}
\sm{\partial Q}
\intd\lambda
\\
&\leq d2^{d+2}
\sum_{P:\exists Q\in\Q,P\subsetneq Q}
\int_{\msl P{\p{\bigcup\Q}}{1/2}}^{f_P}
\sum_{Q\in\Q:Q\supsetneq P}
\f{\lm{\{f>\lambda\}\cap P}}{\sle(Q)}
\intd\lambda
\\
&\lesssim
\sum_{P:\exists Q\in\Q,P\subsetneq Q}
\int_{\msl P{\p{\bigcup\Q}}{1/2}}^{f_P}
\sm{\partial\{f>\lambda\}\cap\ti P}
\intd\lambda
\\
&\leq
\int_{-\infty}^\infty
\sm{\partial\{f>\lambda\}\cap\bigcup\{\ti P:P\in\Q\}}
\intd\lambda
\\
&=
\var_{\{\ti P:P\in\Q\}} f
.
\end{align*}
This finishes the proof.
\end{proof}

%% file: approximation.tex
We need an approximation result to conclude \Cref{theo_goal} from \Cref{pro_goal}.
\begin{lem}[Theorem~5.2 in \cite{MR3409135}]\label{lem_l1approx}
Let \(U\subset\mathbb{R}^d\) be an open set and \(f\in L^1_\loc(U)\).
Let \((f_n)_n\) be a sequence functions that converges to \(f\) in \(L^1_\loc(U)\).
Then
\[\var_Uf\leq\liminf_{n\rightarrow\infty}\var_Uf_n.\]
\end{lem}

\begin{pro}\label{pro_levelsetgoal}
Let \(\Omega\) be open and \(f\in L^1_\loc(\Omega)\) with \(\var_\Omega f<\infty\).
Then
\[\int_\mathbb{R}\sm{\mb{\{x\in\Omega:\M_\Omega f>\lambda\}}}\lesssim\var_\Omega f.\]
The same holds true for \(f\in L^1_\iloc(\Omega)\) and \(\Mi_\Omega f\).
\end{pro}

\Cref{pro_levelsetgoal} is almost \Cref{theo_goal}.
But we also need \(\M_\Omega f\in L^1_\loc(\Omega)\)
to invoke the coarea formula \Cref{lem_coareabv}.

\begin{proof}[Proof of \Cref{pro_levelsetgoal}]
Take an enumeration \(Q_1,Q_2,\ldots\) of all dyadic cubes whose closure/interior is contained in \(\Omega\),
and such that for each \(n\in\mathbb{N}\), each dyadic cube \(Q\subset Q_1\cup\ldots\cup Q_n\) that contains a cube in \(\{Q_1,\ldots,Q_n\}\) already belongs to \(\{Q_1,\ldots,Q_n\}\).
Denote \(\Q^n=\{Q_1,\ldots,Q_n\}\).
Then \((\bigcup\Q^n)_n\) is an increasing sequence of sets.
Since \(\Omega\) is open we have
\[\bigcup_n\ti{\bigcup\Q^n}=\Omega.\]
For a function \(g\) and \(N\in\mathbb{N}\) denote the truncation of \(g\) by
\[\trunc Ng(x)=\min\{\max\{g(x),-N\},N\}.\]
Then for each \(N\in\mathbb{N}\) we have \(\trunc N{\M_\Omega f}\in L^1_\iloc(\Omega)\).
Thus by \Cref{lem_coareabv}
\begin{align*}
\int_\mathbb{R}\sm{\mb{\{x\in\Omega:\M_\Omega f>\lambda\}}}&=\lim_{N\rightarrow\infty}\int_{-N}^N\sm{\mb{\{x\in\Omega:\M_\Omega f>\lambda\}}}\\
&=\lim_{N\rightarrow\infty}\var_\Omega\trunc N{\M_\Omega f}\\
&=\lim_{N\rightarrow\infty}\lim_{n\rightarrow\infty}\var_{\ti{\bigcup\Q_n}}\trunc N{\M_\Omega f}.
\end{align*}
Furthermore \(\trunc N{\M_\Omega f}\) is the pointwise supremum of the set of countable functions \((\trunc N{f_{Q_n}})_n\)
and \(\trunc N{\M_\Omega f}\geq \trunc Nf \) a.e.\ on \(\Omega\).
Thus by monotone convergence \(\trunc N{\M_{\Q^k}f}\) converges to \(\trunc N{\M_\Omega f}\) in \(L^1_\iloc(\Omega)\) for \(k\rightarrow\infty\),
and hence for each \(n\) in \(L^1(\ti{\bigcup\Q^n})\).
Thus by \Cref{lem_l1approx}
\begin{align*}
\var_{\ti{\bigcup\Q_n}}\trunc N{\M_\Omega f}&\leq\liminf_{k\rightarrow\infty}\var_{\ti{\bigcup\Q_n}}\trunc N{\M_{\bigcup\Q^k}f}\\
&\leq\liminf_{k\rightarrow\infty}\var_{\ti{\bigcup\Q_k}}\M_{\bigcup\Q^k}f.
\end{align*}
The above up to here also holds verbatim for \(\Mi_\Omega\) in place of \(\M_\Omega\).
In the setting of \(\M_\Omega\) we have that \(\ti{\bigcup\Q_k}\) is compactly contained in \(\Omega\).
Thus \(f\in L^1(\ti{\bigcup\Q_k})\).
In the setting of \(\Mi_\Omega\) we have \(f\in L^1_\iloc(\Omega)\) so that \(f\in L^1(\ti{\bigcup\Q_k})\) follows from \(\bigcup\Q_k\) being bounded.
Thus in both settings we can invoke \Cref{pro_goal} and get
\[\var_{\ti{\bigcup\Q_k}}\M_{\bigcup\Q^k}f\lesssim\var_\Omega f,\]
which finishes the proof.
\end{proof}

\subsection{Membership of \texorpdfstring{\(L^1_\loc(\Omega)\)}{L1\_loc(Omega)}}

\begin{lem}\label{lem_mflinfae}
Let \(Q_1,Q_2,\ldots\) be an increasing sequence of cubes
and denote by \(Q_\infty\) the quadrant \(Q_1\cup Q_2\cup\ldots\).
Let \(f\in L^1_\loc(Q_\infty)\) with \(f\geq0\), \(\int_{Q_1}f<\infty\) and \(\var_{\ti{Q_\infty}} f<\infty\).
Then
\[\limsup_{n\rightarrow\infty}f_{Q_n}<\infty.\]
\end{lem}

\begin{proof}
Since \(f_{Q_1}<\infty\) it suffices to bound \(|f_{Q_n}-f_{Q_1}|\) independent of \(n\).
By the triangle inequality
and by Poincar\'e's inequality
\begin{align*}
|f_{Q_1}-f_{Q_n}|
&\leq\lm{Q_1}^{-1}\|f-f_{Q_n}\|_{L^1(Q_1)}\\
&\leq\lm{Q_1}^{-\f{d-1}d}\|f-f_{Q_n}\|_{L^{\f d{d-1}}(Q_1)}\\
&\leq\lm{Q_1}^{-\f{d-1}d}\|f-f_{Q_n}\|_{L^{\f d{d-1}}(Q_n)}\\
&\leq\lm{Q_1}^{-\f{d-1}d}\var_{\ti{Q_d}} f\\
&\leq\lm{Q_1}^{-\f{d-1}d}\var_{\ti{Q_\infty}} f\\
&<\infty.
\end{align*}
\end{proof}

\begin{proof}[Proof of \Cref{theo_goal}]
Let \(f\in L^1_\loc(\Omega)\) with \(\var_\Omega f<\infty\).
We have \(|\M_\Omega f|\leq\M_\Omega|f|\)
and \(\var_\Omega|f|\leq\var_\Omega f<\infty\).
Hence it suffices to consider the case that \(f\geq 0\).
Let \(Q\) be a dyadic cube with \(\tc Q\subset\Omega\).
Then by \Cref{lem_mflinfae}
\[c_Q=\sup_{Q\subset\tc P\subset\Omega}f_P<\infty\]
and on \(Q\) we have
\(\M_\Omega f=\max\{\M_Qf,c_Q\}\).
Thus by the maximal function theorem
\begin{align*}
\int_Q(\M_\Omega f)^{\f d{d-1}}&\leq\lm Qc_Q^{\f d{d-1}}+\int_Q(\M_Qf)^{\f d{d-1}}\\
&\leq\lm Qc_Q^{\f d{d-1}}+C_d\int_Qf^{\f d{d-1}}
\end{align*}
which is finite by Sobolev embedding.
Since every set \(U\) which is compactly contained in \(\Omega\)
can be covered by finitely many dyadic cubes which are compactly contained in \(\Omega\),
this implies \(\M_\Omega f\in L^{\f d{d-1}}_\loc(\Omega)\subset L^1_\loc(\Omega)\).
Thus \Cref{lem_coareabv} allows us to invoke \Cref{pro_levelsetgoal} and we get
\[\var_\Omega\M_\Omega f=\int_\mathbb{R}\sm{\mb{\{x\in\Omega:\M_\Omega f>\lambda\}}}\lesssim\var_\Omega f.\]
\end{proof}

\begin{rem}
For \(f\in L^1_\iloc(\Omega)\subset L^1_\loc(\Omega)\) the above arguments also show
\(\Mi_\Omega f\in L^{\f d{d-1}}_\loc(\Omega)\)
and
\(\var_\Omega\Mi_\Omega f\lesssim\var_\Omega f\).
\end{rem}
We need more careful arguments to prove
\(\Mi_\Omega f\in L^1_\iloc(\Omega)\)
and finish the proof of \Cref{theo_goalint}.

\subsection{Membership of \texorpdfstring{\(L^1_\iloc(\Omega)\)}{L1\_tildeloc(Omega)}}

Here we prove that \(\Mi_\Omega f\in L^1_\iloc(\Omega)\) if \(f\in L^1_\iloc(\Omega)\).
Note that for \(f\in L^1_\loc(\Omega)\), with minor tweaks these arguments can be used as an alternative proof of \(f\in L^1_\loc(\Omega)\).

\begin{lem}\label{lem_mfinl1q}
Let \(Q\) be a dyadic cube and \(f\in L^1(Q)\) with \(\var_Q f<\infty\) and \(f\geq0\).
Then
\[\int_Q\M_Q f\lesssim\int_Q f+\lm Q^{\f1d}\var_{\ti Q}f.\]
\end{lem}

\begin{proof}
We use
the boundedness of the maximal operator on \(L^{\f d{d-1}}\)
and
the Gagliardo-Nirenberg-Sobolev inequality \cite[Theorem~5.10~(ii)]{MR3409135} for cubes.
Split
\[\int_Q\M_Q f=\int_Q f+\int_Q\M_Q(f-f_Q)\]
and estimate
\begin{align*}
\avint_Q\M_Q(f-f_Q)
&\leq\Bigl(\avint_Q\M_Q(f-f_Q)^{\f d{d-1}}\Bigr)^{\f{d-1}d}\\
&\lesssim\Bigl(\avint_Q(f-f_Q)^{\f d{d-1}}\Bigr)^{\f{d-1}d}\\
&\lesssim\lm Q^{\f1d-1}\var_Qf.
\end{align*}
\end{proof}

\begin{lem}\label{pro_mfinl1loc}
Let \(\Omega\) be open and \(f\in L^1_\iloc(\Omega)\) with \(\var_\Omega f<\infty\).
Then \(\Mi_\Omega f\in L^1_\loc(\Omega)\).
\end{lem}

\begin{proof}
We have \(|\Mi_\Omega f|\leq\Mi_\Omega|f|\)
and \(\var_\Omega|f|\leq\var_\Omega f<\infty\).
Hence it suffices to consider the case that \(f\geq 0\).
Let \(U\subset\Omega\) be open and bounded and depending on the setting with \(\tc U\subset\Omega\).
We have to show that \(\int_U\Mi_\Omega f<\infty\).
For each \(\lambda\in\mathbb{R}\) let \(\Q_\lambda'\) be the set of cubes \(Q\) that intersect \(U\) and have \(f_Q>\lambda\).
Assume that for each \(\lambda\in\mathbb{R}\) \(\bigcup\Q_\lambda'\) is unbounded.
Since \(U\) is bounded this means for each \(n\) \(\Q_n'\) contains a nested sequence of cubes whose union is a quadrant.
From those we can take a nested diagonal sequence \((Q_n)_n\) whose union is a quadrant and with \(f_{Q_n}\rightarrow\infty\).
But this contradicts \Cref{lem_mflinfae}.
Hence there is a \(\lambda_0\) such that \(\bigcup\Q_{\lambda_0}'\) is bounded.
For each \(\lambda\in\mathbb{R}\) let \(\Q_\lambda\) be the set of maximal cubes in \(\Q_\lambda'\).
Then for \(\lambda\geq\lambda_0\) we have \(\bigcup\Q_\lambda'=\bigcup\Q_\lambda\) and
\begin{align*}
\int_U\Mi_\Omega f&=\int_0^\infty\lm{\{x\in U:\Mi_\Omega f>\lambda\}}\intd\lambda\\
&\leq\int_0^{\lambda_0}\lm U\intd\lambda+\int_{\lambda_0}^\infty\sum_{Q\in\Q_{\lambda_0}}\lm{\{x\in Q:\Mi_\Omega f>\lambda\}}\intd\lambda\\
&\leq\lambda_0\lm U+\sum_{Q\in\Q_{\lambda_0}}\int_{\lambda_0}^\infty\lm{\{x\in Q:\Mi_\Omega f>\lambda\}}\intd\lambda
\end{align*}
Since \(\lm U<\infty\) it suffices to bound the second term in the previous display.
For \(x\in Q\) with \(\Mi_\Omega f(x)>\lambda_0\) we have \(\Mi_\Omega f(x)=\Mi_Qf(x)\).
Thus by \Cref{lem_mfinl1q} we have
\begin{align*}
\sum_{Q\in\Q_{\lambda_0}}\int_{\lambda_0}^\infty\lm{\{x\in Q:\Mi_\Omega f>\lambda\}}\intd\lambda
&\lesssim\sum_{Q\in\Q_{\lambda_0}}\int_Q f+\sum_{Q\in\Q_{\lambda_0}}\lm Q^{\f1d}\var_{\ti Q}f\\
&\leq\int_{\bigcup\Q_{\lambda_0}}f+\lm{\bigcup\Q_{\lambda_0}}^{\f1d}\var_{\ti{\bigcup\Q_{\lambda_0}}}f.
\end{align*}
This is finite because
\(\lm{\bigcup\Q_{\lambda_0}}<\infty\),
\(f\in L^1_\iloc(\Omega)\)
and \(\var_\Omega f<\infty\).
\end{proof}

\begin{proof}[Proof of \Cref{theo_goalint}]
By \Cref{pro_mfinl1loc} we have \(\Mi_\Omega f\in L^1_\iloc(\Omega)\).
Thus \Cref{lem_coareabv} allows us to invoke \Cref{pro_levelsetgoal} and we get
\[\var_\Omega\M_\Omega f=\int_\mathbb{R}\sm{\mb{\{x\in\Omega:\M_\Omega f>\lambda\}}}\lesssim\var_\Omega f.\]
\end{proof}

%% file: prospects.tex
Now maybe the most obvious strategy to prove
\begin{equation}\label{eq_goal}
\var\M f\leq C_d\var f
\end{equation}
for the uncentered Hardy-Littlewood maximal operator is to transfer the arguments of this paper from dyadic cubes to balls, using ideas from \cite{weigt2020variation}.
So here is a potential alternative proof strategy of \cref{eq_goal}.
The idea is to conclude \cref{eq_goal} for general functions \(f\)
from \cref{eq_goal} for characteristic functions and subadditivity.

\begin{cla}\label{cla_charftogen}
Let \(\M\) be a maximal operator.
Assume that there is a functional \(V:\BV(\mathbb{R}^d)\rightarrow\mathbb{R}\) with the following properties.
For a function \(g\) with bounded variation and a characteristic function \(h\) supported on \(\{x:g(x)=\|g\|_\infty\}\)
\(V\) is subadditive,
\begin{equation}\label{eq_vsublinear}
V(g+h)\leq V(g)+V(h)
\end{equation}
and 
\begin{equation}\label{eq_vcharf}
V(h)\leq C_d\var h.
\end{equation}
Furthermore for all for all \(f\in\BV\) we have
\begin{equation}\label{eq_varmleqv}
\var\M f\leq V(f).
\end{equation}
Then we can conclude also for all \(f\in\BV\) that
\[\var\M f\leq C_d\var f.\]
\end{cla}

\begin{proof}[Proof sketch]
Let \(h_1\geq\ldots\geq h_n\) be characteristic functions.
Then inductively we have
\begin{align*}
\var\M(h_1+\ldots+h_n)&\leq V(h_1+\ldots+h_n)\\
&\leq V(h_1+\ldots+h_{n-1})+V(h_n)\\
&\ldots\\
&\leq V(h_1)+\ldots+V(h_n)\\
&\leq C_d\var h_1+\ldots+\var h_n\\
&=C_d\var(h_1+\ldots+h_n),
\end{align*}
and by approximation we can extend the estimate from sums of characteristic functions to general \(f\in\BV\).
\end{proof}

Note that \(V=C_d\var\) satisfies \cref{eq_vsublinear,eq_vcharf},
but the whole statement becomes trivial because then \cref{eq_varmleqv} is already what we want to prove.
Still, this shows that the existence of a \(V\) as in \Cref{cla_charftogen} is actually equivalent to \cref{eq_goal}.
However even for the dyadic maximal operator the only \(V\) we have found is \(C_d\var\).

Another candidate for \(V\) is \(\var\M\),
because we already know \cref{eq_vcharf} from \cite{weigt2020variation}
and \cref{eq_varmleqv} is trivial.
In  \cite{weigt2020variation} we even prove \cref{eq_vcharf} for the uncentered Hardy-Littlewood maximal operator.
But \cref{eq_vsublinear} unfortunately fails for both the dyadic and the uncentered Hardy-Littlewood operator,
see \Cref{exa_nosublinearity}.

\begin{exa}\label{exa_nosublinearity}
Let \(d=1\) and \(g=\ind{[0,3)}+\ind{[5,8)}\) and \(h=\ind{[2,3)}+\ind{[5,6)}\).
Then for \(f=g,h,g+h\) we have
\[\var\M f=\M f(2.5)+[\M f(2.5)-\M f(4)]+[\M f(5.5)-\M f(4)]+\M f(5.5).\]
Thus
\begin{align*}
\var\M g&=1+(1-3/4)+(1-3/4)+1=2.5,\\
\var\M h&=1+(1-1/2)+(1-1/2)+1=3,\\
\var\M(g+h)&=2+(2-1)+(2-1)+2=6>3+2.5.
\end{align*}

The same counterexample works for the uncentered Hardy-Littlewood maximal function.
It also works for \(d>1\) by defining \(\tilde g:\mathbb{R}^d\rightarrow\mathbb{R}\) by \(\tilde g(x)=g(x_1)\ind{[-N,N]^d}(x)\)
with \(N\) large enough
and \(\tilde h:\mathbb{R}^d\rightarrow\mathbb{R}\) similarly.
\end{exa}

However since maximal operators are pointwise subadditive,
one might hope to find a modification of \(\var\M\) that is subadditive.
The most promising candidate is \(V(f)=\max\{\var\M f,C_d\var f\}.\)
By \cite{weigt2020variation} it satisfies \cref{eq_vcharf}
and it clearly also satisfies \cref{eq_varmleqv}.
In order to prove \cref{eq_vsublinear},
by the subadditivity of \(\var\) it suffices to prove
that there is a \(C_d\) such that
\begin{equation}\label{eq_subaddimproved}
\var\M(g+h)\leq\var\M g+C_d\var h
\end{equation}
for \(g,h\) as in \Cref{cla_charftogen}.
That means from \cref{eq_subaddimproved} one could conclude \(\var\M f\leq C_d\var f\).
But even for the dyadic maximal operator \cref{eq_subaddimproved} is still an open question.
Note that in \Cref{exa_nosublinearity} we have \(\var h=4\) so it is not a counterexample against \cref{eq_subaddimproved}.
Maybe \cref{eq_subaddimproved} is still close enough to \(\var\M h\leq C_d\var h\) for characteristic functions
to make use of \cite{weigt2020variation}.
One might also come up with more sophisticated functionals in between \(\var\M\) and \(\var\)
that satisfy the premise of \Cref{cla_charftogen}.